\documentclass[12pt]
{article}
\usepackage{amsmath,amsthm,amssymb,color}
\usepackage{times}
\usepackage{array}
\usepackage{enumerate}

\def\d{\partial}

\def\qed{\hfill$\Box$}
\def\bR{\mathbb R}
\def\bN{\mathbb N}
\def\bT{\mathbb T}
\def\ep{\epsilon}

\def\vfi{\varphi}

\def\sgn{\hbox{sgn}\,}
\def\supp{\hbox{supp}\,}

\def\cE{{\cal E}}
\def\cF{{\cal F}}
\def\cG{{\cal G}}

\def\cL{{\mathcal L}}

\theoremstyle{definition}
\numberwithin{equation}{section}

\newtheorem{proposition}{Proposition}[section]
\newtheorem{theorem}{Theorem}[section]

\newtheorem{lemma}{Lemma}[section]
\newtheorem{corollary}{Corollary}[section]

\newcommand{\subjclass}[1]{\bigskip\noindent\emph{2010 Mathematics Subject Classification:}\enspace#1}
\newcommand{\keywords}[1]{\noindent\emph{Keywords:}\enspace#1}

\begin{document}


\title{\bf Integrability of the derivative of solutions to a singular one-dimensional parabolic problem}

\author{Atsushi~Nakayasu \\
 Graduate School of Mathematical Science, University of Tokyo\\
  Komaba 3-8-1, Meguro-ku, Tokyo 153-8914, Japan\\
{\tt ankys@ms.u-tokyo.ac.jp}\\
  Piotr~Rybka \\
 Institute of Applied Mathematics and Mechanics,
 Warsaw University\\ ul. Banacha 2, 07-097 Warsaw, Poland\\
 {\tt rybka@mimuw.edu.pl}}

\maketitle
\date{}

{\it To the memory of Professor Marek Burnat}

\bigskip

\begin{abstract}We study integrability of the derivative of solutions to a singular one-dimensional parabolic equation with
initial data in $W^{1,1}$. In order to avoid additional difficulties we consider only the periodic boundary conditions. The problem we study is a gradient flow of a convex, linear growth variational functional. We also prove a similar result for the elliptic companion problem, i.e. the time semidiscretization.


\end{abstract}

\section{Introduction}
We study a one-dimensional parabolic equation,
\begin{equation}\label{ir1}
 \begin{array}{ll}
  u_t = (W_p( u_x) )_x, &(x,t)\in Q_T: =\bT\times (0,T), \\
  u(x, 0) = u_0(x), & x\in \bT,
 \end{array}
\end{equation}
where $\bT$ is a flat one-dimensional torus, which we identify with $[0,1)$. In other words, for the sake of simplicity we consider the periodic boundary conditions, but the same argument with little change applies to zero Neumann data.

Eq. (\ref{ir1}) is formally a gradient flow of the following functional,
$$
\cE(u)= \left\{
\begin{array}{ll}
\int_\bT W(u_x)\,dx, & u \in W^{1,1}(\bT),\\
+\infty & u \in L^2(\bT)\setminus W^{1,1}(\bT).
\end{array}
\right. 
$$
Our main assumption on $W$, apart from convexity, is the linear growth of $W$.

We also consider a companion of this equation, namely, the time semidiscretization of (\ref{ir1}),
\begin{equation}\label{iE}
 \frac1h (u-f) = (W_p(u_x))_x\quad\hbox{in }\bT.
\end{equation}

Even though it makes sense to consider $u_0\in BV$ for eq. (\ref{ir1}) we study here the propagation of regularity, i.e. we show that integrability of $\frac d{dx} u_0$ (denoted by $u_{0,x}$) implies that the derivative of the weak solution is also integrable, $u_x(\cdot, t)\in L^1$, see Theorem \ref{main1} in Section \ref{ser}. Apparently, such results are not known in the general context. We are only aware of the paper by Bellettini {\it et al.}, \cite{bellettini}, on the parabolic minimal surface equation, for which the authors
show that the solutions are eventually regularized, i.e. there is a positive waiting time. We stress that our assumptions on $W$ are more general, since we need only convexity
and the linear growth. The precise formulation of these conditions is in the statement of Theorem \ref{main1}.

What we prove in Theorem \ref{main1} shows that eq. (\ref{ir1}) does not create singularities like jumps. Such a result is known in a multidimensional setting for 
$W(p) =|p|$. In particular, the jumps present in the data persist, see \cite{caselles2007}, and H\"older continuity of the data propagates, \cite{caselles2011}. We also note that our method is essentially restricted to one dimension. We are not able to address the same question in higher dimensions.

Our Theorem \ref{main2} is a companion result on a closely related elliptic problem, (\ref{iE}). But we prove it first, because it is slightly simpler than Theorem \ref{main1}. Here, in eq. (\ref{iE}) $f$ plays the role of initial conditions, hence $f\in L^p$, $p\ge1$ implies only that $u\in BV$. Since eq. (\ref{iE}) is the time semidiscretization of (\ref{ir1}), then integrability of the derivative of solutions following from integrability of the derivative of $f$ is not surprising. A similar statement for a domain in $\bR^N$ is proved by Beck {\it et al.} in \cite{bulicek}, but for smooth nonlinearities corresponding to functionals with linear growth. In the setting of \cite{bulicek} the smooth dependence of $W$ on $p$ is important for the argument. In \cite{maringova}, in a similar setting Lipschitz continuity of minimizers is shown. 

If $W(p) =|p|$, the we can offer an additional comment about solutions to (\ref{iE}), which is the Euler-Lagrange eq. for the Rudin-Osher-Fatemi functional, see \cite{rudin-etco}. We can say that if data are regular, in this case $f\in W^{1,1}$, then we cannot detect edges, understood as jumps of $u$ solutions to (\ref{iE}), because jumps may not be created.

Both of our results can be expressed as no singularity formation. They are both obtained with the same technique depending on the insight into the structure of $L^1$. The necessary preliminary results are presented in Section \ref{spre}. Namely, if  function $g$ belongs to $L^1$, then it automatically enjoys a better integrability, see Lemma \ref{fi} and \cite[\S 2.1]{RaoRen}. In our setting $g$ is the derivative of data, i.e. $g=f_x$ in the case of equation (\ref{iE}) or $g=(u_0)_x$ in the case of parabolic equation (\ref{ir1}). In fact, we show that this better integrability of derivatives of data is passed to the derivatives of solutions, see Theorem \ref{main1} and Theorem \ref{main2}.

We show first the desired estimates for solutions to the regularized problems either elliptic or parabolic. 
The passage to the limit requires weak compactness in $L^1$ and the Pettis theorem.
In order to show that the limit of solutions to the regularized problems are actually solutions to the original equation we depend on the theory of monotone operators, i.e.  Minty's trick.

In  Section \ref{sell} we prove first our result for the elliptic problem. For this purpose we study
solutions to a regularized problem. 
The parabolic problem, treated in Theorem \ref{main1}, requires an additional step, as compared with the elliptic equation, and this is why we deal with this in the last section. Section \ref{ser} is closed with a remark on finite extinction (or rather stopping) time, which is common to the problems we consider, if $W$ has a singularity at $p=0$.

\section{Preliminaries}\label{spre}

We gather here our assumptions on $W$ and we present the necessary information about the structure of the space $L^1(\Omega)$ for any $\Omega\subset \bR^N$.

\subsection{Conditions on $W$ and functional $\cE$}

Throughout the paper, we assume that $W$ is an even, convex function with linear growth at infinity, i.e.
\begin{equation}\label{gr}
\lim_{t\to \infty} \frac{W(t)}t = W^+, \qquad \lim_{t\to \infty} \frac{W(-t)}t =W^- . 
\end{equation}
In the above formula, $ W^\pm$ are positive numbers. Without the loss of generality we could assume that 
\begin{equation}\label{gr2}
 W^+ = W^- = W^\infty>0.
\end{equation}
Indeed, one could consider $\tilde W(p) = W(p) +\frac12( W^- - W^+)$ in place of $W$. This modified $\tilde W$ does not change neither (\ref{ir1}) nor (\ref{iE}).


We will not impose any further restrictions $W$. Here are some examples,
$$
|p|, \qquad |p+1|+|p-1|, \qquad \sqrt{1+p^2}, \qquad |p|+\sqrt{1+p^2}.
$$


We notice that functional $\cE$ is defined naturally on the space $W^{1,1}$. 
However, in general $\cE$ is not lower semicontinuous on $W^{1,1}$ with respect to the $L^2$ topology, unless $W$ is piecewise linear, see \cite{mm}, \cite{nr}. The lower semicontinuous envelope or the relaxation of $\cE$, denoted by $\bar\cE$, is naturally well-defined on $BV(\bT)$.
For $u\in BV$ we write,
\begin{equation}\label{defbe}
\bar \cE(u) = \inf\{ \varliminf_{n\to \infty} \cE(u_n): \ u_n \to u\hbox{ in } L^2\}.
\end{equation}
We know that (see \cite[Theorem 5.47]{AFP}), 
\begin{equation}\label{zero}
\bar \cE(u) = \int_\bT W(u_x)\,dx + W^\infty \int_\bT | D^s u|. 
\end{equation}
Here, $D u =  u_x${\Large$\llcorner$}$ \cL^1 + D^s u$ is a decomposition of measure $Du$ into an absolutely continuous part with respect to the Lebesgue measure and a part singular to it.

\subsection{The useful structure of $L^1$}
Here, we recall the necessary information on $L^1$ needed to derive our estimates on solutions to (\ref{ir1}) and (\ref{iE}).

\begin{lemma}\label{fi}
 Let us suppose that $f\in L^1(\Omega)$, then there exists a smooth, convex function $\Phi:\bR\to\bR$ such that, $\lim_{|x|\to \infty} \Phi(x)/|x| =\infty$ and
\begin{equation}\label{fifi}
 \int_\Omega \Phi(f)\,dx <\infty.
\end{equation}
\end{lemma}
\begin{proof} By \cite[\S 1.2, Corollary 3]{RaoRen}, we know that there exists a convex function $\tilde \Phi$ such that 
$$
\lim_{|x|\to \infty} \tilde\Phi(x)/|x| =\infty
$$ 
and
$$
 \int_\Omega \tilde\Phi(f)\,dx <\infty.
$$
Now, for all $\delta>0$, we define
$$
\hat \Phi_\delta(p)=
\left\{
\begin{array}{ll}
\tilde \Phi(p-\delta) & p>\delta,\\
\tilde \Phi(p+\delta) & p< -\delta,\\
\tilde \Phi(0) & |p|\le \delta.
\end{array}
\right. 
$$
Once we have it, we
take $\Phi = \hat\Phi_\delta*\phi_\delta$, for any 
$\delta<1,$ where $\phi_\delta$ is the standard, positive mollifier kernel with $\supp \phi_1 \subset B(0,1)$ and $\max \phi = \phi(0)$. It is easy to see that $\Phi(p)/|p| \to +\infty$ as $|p|\to +\infty$. 

Now, we check that
\begin{equation}\label{nie}
 \Phi (p) \le  C_0 \tilde\Phi(p) + C_1,
\end{equation}
where $C_0=\phi(0)/\delta$. 
For $p>1$ we see that
$$
\Phi (p) \le \frac{1}{\delta}\int_\bR \hat \Phi_\delta(q)\phi(\frac{p-q}{\delta})\,dq \le
\frac{\phi(0)}{\delta}\int_{p-\delta}^{p+\delta}  \hat \Phi_\delta(p+\delta)\,dq = C_0\tilde \Phi(p).
$$
Similar inequality holds for $p<-1$. 

If $|p|\le \delta$, then
$$
\Phi(p) \le C_0\tilde \Phi(0) \le \tilde \Phi(p) + C_1,
$$
where $C_1 =  C_0\max\{1, \tilde\Phi(0)\}$.
Thus, (\ref{nie}) holds. Since we established (\ref{nie}), we conclude that (\ref{fifi}) holds too.
\end{proof}

We recall that a family $\cF$ of integrable functions is {\it uniformly integrable} if and only if
$$
\hbox{(i)} \quad \sup_{f\in\cF} \int_\Omega |f| \,d\mu =c<\infty
\quad\hbox{and}\quad \hbox{(ii)} \quad \lim_{\mu(A) \to 0}\int_A |f|\,d\mu
\quad \hbox{uniformly with respect to }f\in \cF.
$$

Let us introduce the notation 
$$\cG(v):= \int_D\Phi(v(x))\,dx,
$$
where $D = \bT$ or $D=Q_T$. The Pettis Theorem immediately implies the following fact.
\begin{lemma}\label{pettis}
 If  a sequence $\cF=\{ f_k\}_{k=0}^\infty\subset L^1(D)$ satisfies
$$
\cG(f_k)\le M <\infty,\quad k\in \bN,
$$
then we can select a subsequence  $f_{k_m}$ converging weakly in $L^1(D)$ to $f\in L^1(D)$. \qed
\end{lemma}

We address now the question of the limit passage in $\cG$ or $\cE$.

\begin{lemma}\label{1.4}
 Let us suppose that $f_n\in L^1(D)$, where $D\subset \bR^d$, $d=1,2$, satisfy the following bound,
 $$
 \int_D \Phi (f_n)\,dx \le M,
 $$ 
 where $\Phi$ is as in Lemma \ref{fi},
 and $f_n \rightharpoonup f$ in $L^1(D)$. Then,
 $$
 \varliminf_{n\to \infty}  \int_D \Phi (f_n(x))\,dx \ge \int_D \Phi (f(x))\,dx.
$$
\end{lemma}
{\it Proof.} 
Due to the convexity of $\Phi$, this function is an envelope of a family of straight lines,
$$
\Phi(p) = \sup_{\alpha\in I} \ell_\alpha (p).
$$
Thus, for any index $\alpha$ we have $\Phi(p) \ge \ell_\alpha (p) = a_\alpha p + b_\alpha$ and
$$
\varliminf_{n\to \infty} \int_D\cG(f_n(x))\,dx \ge \varliminf_{n\to \infty} \int_D \ell_\alpha(f_n(x))\,dx
= \int_D a_\alpha f(x)\,dx + b_\alpha |D|,
$$
because any constant $a_\alpha$ may be identified with a  continuous functional over $L^1$. Thus,
$$
\varliminf_{n\to \infty} \int_D\cG(f_n(x))\,dx \ge  \int_D\ell_\alpha (f(x))\,dx .
$$
After having taken the supremum over $\alpha\in I$ we reach the claim. \qed

\section{The elliptic problem of time semidiscretization}\label{sell}

We first deal with integrability of solutions to the following elliptic problem, 
\begin{equation}\label{E}
 \frac1h (u-f) = (W_p(u_x))_x\quad\hbox{in }\Omega,
\end{equation}
augmented with either periodic or Neumann boundary conditions.
First of all, we have to settle the meaning of a solution to (\ref{E}). If we assume that $f$ is in $L^2$, then (\ref{E}) is formally the Euler-Lagrange equation of the following functional,
$$
\cE(u) + \frac{1}{2h} \int_\Omega (u-f)^2\,dx.
$$
However, due to the lack of lower semicontinuity of $\cE$ in general,
we could understand solutions to (\ref{E}) as minimizers, which are the only critical points here, to
$$
\cF_f (u) = \bar\cE(u) + \frac{1}{2h} \int_\Omega (u-f)^2\,dx,
$$
where $\bar\cE$ is the lower semicontinuous envelope of $\cE$ defined in (\ref{defbe}), cf.
(\ref{zero}). In this case, we notice.

\begin{corollary}
 If $W$ is convex, the assumptions (\ref{gr}), (\ref{gr2}) hold and for all $p\in \bR$ we have $W(p) \ge \alpha |p|$ and $u$ is a minimizer of $\cF_f$, then
$$
|Du|(\Omega) \le \frac{1}{\alpha }\|f\|^2_{L^2},\qquad
\| u\|_{L^2} \le 4\|f\|^2_{L^2}.
$$
\end{corollary}
\begin{proof}
 This is an immediate conclusion from  $\cF_f(u) \le \cF_f(0)$.
\end{proof}

However, this simplistic approach is not sufficient to deduce that $u_x \in L^1$. If we wish to establish integrability of the derivative of the solution to (\ref{E}), we have to proceed differently. Since we expect that $u\in W^{1,1}$, 
we can define the appropriate notion of a solution. We say that a function $u\in W^{1,1}$ is a {\it weak solution} to (\ref{E}) if there
exists $\xi\in L^\infty$, $\xi_x\in L^2$ such that $\xi(x) \in \partial W(u_x(x))$ for a.e. $x\in \bT$ and the following identity
$$
\int_\bT (\frac 1h (u-f)\vfi + \xi \vfi_x)\,dx =0
$$
holds for all $\vfi \in C^\infty(\bT)$. We notice that since $C^\infty(\bT)$ is dense in $W^{1,1}(\bT)$ and $W^{1,1}(\bT)\subset L^2(\bT)$, we can take test functions from $W^{1,1}$.

We prove:
\begin{theorem}\label{main2}
Let us assume that $W$ is convex, the assumption (\ref{gr}--\ref{gr2}) 
holds.
If $f\in W^{1,1}$ and $h>0$, then there exists a unique solution to (\ref{E}), $u$, which has the  integrable derivative. Moreover,
\begin{equation}\label{kees}
 \cG(u_x) \le \cG(f_x),
\end{equation}
$\Phi$ is given by Lemma \ref{fi} for $f_x$.
\end{theorem}\noindent
{\it Proof. \ }
In order to obtain existence of solutions, we
regularize the equation by adding the $\ep u_{xx}$ term and smoothing out the nonlinearity, 
$W^\ep (p) = (W *\rho_\ep)(p)$, where $\rho_\ep$ is the standard symmetric mollifying kernel.
Thus, we consider,
 \begin{equation}\label{Ee}
  \frac1h( u^\ep -f) = W^\ep_p(u^\ep_x)_x + \ep u^\ep_{xx}\qquad x\in \bT.
 \end{equation}
We shall say that a function $u^\ep \in W^{1,2}(\bT)$ is a {\it weak solution} to (\ref{Ee}) if 
following identity holds
\begin{equation}\label{ELee}
\int_\bT (\frac 1h (u^\ep-f)\vfi + (W^\ep_p(u^\ep_x) +\ep u^\ep_x) \vfi_x)\,dx =0
\end{equation}
for all $\vfi \in C^\infty(\bT)$. In formula (\ref{ELee}) we require that $\varphi$ is smooth, but since $C^\infty(\bT)$ is dense in $W^{1,2}(\bT)$ we may use $u^\ep$ as a test function.

We notice that equation (\ref{ELee}) is the Euler-Lagrange eq. for the functional
$$
\cF^\ep_f(u) = \int_\bT(\frac 1{2h} (u-f)^2 + W^\ep(u_x) +\frac\ep 2 u_x^2)\,dx.
$$
Since $\cF^\ep_f$ is strictly convex and lower semicontinuous on $W^{1,2}(\bT)$,
we immediately conclude existence and uniqueness of minimizers, $u^\ep\in W^{1,2}(\bT)$. 
Since $W^\ep$ is smooth, we immediately conclude that $u^\ep$ satisfies (\ref{ELee}).

Due to the linear growth of $W$ the derivative $W^\ep_{pp}$ is bounded and  $W^\ep_{pp}+\ep\ge \ep$. Hence it is easy to deduce  higher regularity of $u^\ep$, i.e. $u^\ep\in W^{2,2}(\bT)$, because
$$
\frac 1h (u^\ep-f) = (W^\ep_{pp}(u_x^\ep)+\ep)u_{xx}^\ep.
$$

We set 
\begin{equation}\label{xip}
\xi^\ep = W^\ep_{p}(u_x^\ep),
\end{equation}
we notice that $\xi^\ep\in W^{1,2}(\bT)$. Since $W^\ep$ is convex, then its derivative is a monotone function. If we combine it with the linear growth of $W$, then we notice,
\begin{equation}\label{xid}
\xi^\ep (x) \in [-W^\infty, W^\infty].
\end{equation}

We have to deduce that the family $\{ u^\ep\}$ is relatively weakly compact in $W^{1,1}(\bT)$. The main point is establishing existence of a subsequence $\{ u^\ep_x\}$ converging weakly in $L^1$.
For this purpose, we use Lemma \ref{fi} guaranteeing that (\ref{fifi}) holds, i.e. $\cG(f_x)<\infty$.
Once we have $\Phi$,
we multiply both sides of (\ref{Ee}) by $\Phi''(u^\ep_x) u^\ep_{xx} \in L^2(\bT).$ After integration over $\Omega$ and integration by parts, we come to
 $$
 \int_\Omega (f_x \Phi'(u^\ep_x) - u^\ep_x \Phi'(u^\ep_x))\ge 0.
 $$
Now, convexity of $\Phi$ gives us,
 $$
 \int_\Omega \Phi(f_x)\,dx - \int_\Omega \Phi(u^\ep_x) \,dx \ge 
 \int_\Omega\Phi'(u^\ep_x) (f_x-u^\ep_x).
 $$
Combining these two inequalities yields,
 \begin{equation}\label{Gep}
  \cG(u^\ep_x)\equiv  \int_\Omega \Phi(u^\ep_x) \le \int_\Omega \Phi(f_x)\,dx\equiv 
\cG(f_x).
 \end{equation}
Now, we can use Lemma \ref{pettis} to deduce the weak convergence in $L^1$ of $u_x^\ep$ to $u_x\in L^1$ as $\ep \to 0$.
In the next step,   Lemma \ref{1.4} guarantees the lower semicontinuity of $\cG$ and $\cE$ with respect to weak convergence in $L^1$. Thus, we reach the bound (\ref{kees}).

Now, we want to show that $u$ is indeed a weak solution to (\ref{E}), i.e. we have to find $\xi$ stipulated by the definition of a weak  solution and to show that it has the desired properties. For each $\ep>0$ we have at our disposal, solutions $u^\ep$ to (\ref{ELee}) and $\xi^\ep$ defined by (\ref{xip}). 
We notice that due to (\ref{xid}) $\xi^\ep$ converges (possibly after extracting a subsequence) weakly${}^*$ in $L^\infty$ to $\xi$ and $\xi(x)\in [-W^\infty,W^\infty]$ a.e.

We know that 
$u^\ep_x$ converges weakly in $L^1$ and we assumed that the test function $\vfi$ in (\ref{ELee}) is in $C^\infty$. Thus, in order to be able to can pass to the limit in (\ref{ELee}) we need to know that $\ep \int_\bT u^\ep_x \varphi_x \,dx$ goes to zero as $\ep\to 0$. Indeed, since $u^\ep$ is a minimizer of $\cF^\ep_f$, then we notice
$$
\ep \|u^\ep_x\|_{L^2}^2 \le \cF^\ep_f(u^\ep) \le \cF^\ep_f(0) = \frac 1h \int_\bT [f^2 + W(0)]\,dx + \frac 1h=: C_\cF.
$$
Thus, 
$$
\ep \left|\int_\bT u^\ep_x \varphi_x \,dx \right| 
\le \ep \|u^\ep_x\|_{L^2} \|\varphi_x\|_{L^2} \le \ep^{1/2} \sqrt{C_\cF} \to 0.
$$
Finally, after passing to the limit in (\ref{ELee}), we obtain the following identity,
\begin{equation}\label{rWe}
\int_\bT (\frac 1h (u-f)\vfi + \xi\vfi_x)\,dx =0
\end{equation}
for all $\vfi \in C^\infty(\bT)$. 
The density of $C^\infty$ in $W^{1,1}$ and the embedding $W^{1,1}\subset L^2$ imply that we may take test functions from $W^{1,1}$ in (\ref{rWe}). 

It is important to notice that (\ref{rWe}) implies that $\xi\in W^{1,2}$. Indeed, 
due to (\ref{rWe}) the weak derivative of $\xi$ is $\frac 1h (u-f)$, hence our claim follows.

Now, it remains to show that 
$\xi(x) \in \d W(u_x(x))$ for almost every $x\in\bT$. Indeed, from the construction of $u^\epsilon$ we know that for any $w\in W^{1,1}$ we have
\begin{equation}\label{rdop}
 \int_\bT W^\ep(w_x)\,dx \ge  \int_\bT \xi^\epsilon (w_x- u^\epsilon_x)\,dx +\int_\bT W^\ep(u^\epsilon_x)\,dx.
\end{equation}
We want to calculate the limit of both sides taking into account that  
Since,
\begin{equation}\label{zb}
 \xi^\epsilon \stackrel{*}{\rightharpoonup}\xi \hbox{ in } L^\infty (\bT)\qquad\hbox{and}\qquad
u^\epsilon_x {\rightharpoonup} u_x \hbox{ in } L^1(\bT).
\end{equation}
In order to proceed we have to take a close look at each term in (\ref{rdop}). 

Due to the locally uniform convergence of $W^\ep$ to $W$ and the Lebesgue dominated convergence theorem we deduce that
\begin{equation}\label{zb1}
\lim_{\ep\to 0} \int_\bT W^\ep(w_x)\,dx = \int_\bT W(w_x)\,dx.
\end{equation}
Next, we notice that Jensen inequality gives us $W^\ep(p)\ge W(p)$. Hence, Lemma \ref{1.4} yields,
\begin{equation}\label{zb2}
\varliminf_{\ep \to 0}\int_\bT W^\ep(u^\epsilon_x)\,dx \ge 
\varliminf_{\ep \to 0}\int_\bT W(u^\epsilon_x)\,dx \ge \int_\bT W(u_x)\,dx
.
\end{equation}
Finally, we look at 
$\int_\bT \xi^\epsilon u^\epsilon_x$ in (\ref{rdop}). We use (\ref{ELee}), where we take $u^\ep$ for a test function. Thus, we obtain
$$
- \int_\bT \xi^\epsilon u^\epsilon_x =
\int_\bT \ep |u^\epsilon_x|^2 + \frac 1h\int_\bT (u^\ep - f)u^\ep\,dx.
$$
If we use this information, then (\ref{rdop}) takes the following form,
\begin{equation*}
 \int_\bT W^\ep(w_x)\,dx \ge  \int_\bT \xi^\epsilon  w_x\,dx
 + \int_\bT \ep |u^\epsilon_x|^2\,dx  + \frac 1h\int_\bT (u^\ep - f)u^\ep\,dx
 +\int_\bT W^\ep(u^\epsilon_x)\,dx.
\end{equation*}
After dropping the positive term $\int_\bT \ep |u^\epsilon_x|^2\,dx$ on the RHS and taking the liminf, using (\ref{zb}),  (\ref{zb1}) and  (\ref{zb2}), we arrive at
\begin{equation*}
 \int_\bT W(w_x)\,dx \ge  \int_\bT \xi  w_x\,dx
 + \frac 1h\int_\bT (u - f)u\,dx
 +\int_\bT W(u_x)\,dx.
\end{equation*}
We use (\ref{rWe}) again, we reach
\begin{equation}\label{kpod}
 \int_\bT W(w_x)\,dx \ge  \int_\bT \xi ( w_x-u_x)\,dx
 +\int_\bT W(u_x)\,dx.
\end{equation}

Relaying on (\ref{kpod}), 
$u_x\in L^1$, due to Lemma \ref{lpod} below, we  deduce  that $\xi(x)\in \d W(u_x)$ a.e. 
Thus, indeed $u\in W^{1,1}$ is a weak solution to (\ref{E}). Moreover, (\ref{Gep}) and Lemma \ref{1.4} imply that
$$
\int_\bT \Phi(u_x)\,dx \le \int_\bT \Phi(f_x)\,dx . \eqno\Box
$$

Before we state Lemma \ref{lpod} we notice that our argument show that 
\begin{corollary}
 If $u$ is a solution constructed in the previous theorem, then $-\xi_x \in \d \bar \cE(u)$.
\end{corollary}
\begin{proof}
 We will see that $-\xi_x$ is an  element of the subdifferential $\d\bar\cE(u)$. We know that for $u\in W^{1,1}$, it is true that $\cE(u) = \bar\cE(u)$. If $w\in BV$, then $w= v+\psi$, where $w_x = v_x$, $w_x\in L^1$ and $\psi_x = 0$ $\cL^1$-a.e. Then,
$$
\bar\cE(w) = \bar\cE(v+\psi) = \cE(v) + \int_\bT W^\infty |D^s\psi|.
$$
Moreover, $\xi$ the weak${}^*$ limit of $\xi^\epsilon$ with values in $[-W^\infty, W^\infty]$ satisfies the same constraint. Since $D^s \psi = \sigma |D^s\psi|$, where $|\sigma| =1$ $|D^s\psi|$-a.e., then
$$
\int_\bT W^\infty |D^s\psi| - \xi D^s\psi = \int_\bT (W^\infty -  \xi \sigma)|D^s\psi|\ge 0,
$$ 
because $(W^\infty -  \xi \sigma)(x)\ge 0$ for $|D^s\psi|$-a.e. $x\in \bT$.

Combining the available information, we obtain,
\begin{eqnarray*}
 \bar\cE(w) - \bar\cE(u) &=&  \cE(v) - \cE(u) +  \int_\bT W^\infty |D^s\psi|\\
 &\ge & \int_\bT \xi(v_x - u_x)\,dx + \int_\bT\xi D^s\psi \\
 &=& - \int_\bT \xi_x(v - u)\,dx - \int_\bT\xi_x \psi\,dx  = -\int_\bT \xi_x (w-u)\,dx .
\end{eqnarray*}
In other words, $-\xi_x\in\partial\bar \cE(u)$.
\end{proof}

\begin{lemma}\label{lpod}
 Let us assume that $\xi\in W^{1,2}(\bT)$ is such that $\xi(x)\in[-W^\infty, W^\infty]$ and (\ref{kpod}) holds for all $w\in W^{1,1}(\bT)$. Then, $\xi(x)\in \d W(u_x(x))$ for almost all $x\in \bT$.
\end{lemma}
\begin{proof}
 We will construct special test functions $h\in W^{1,1}(\bT)$. For any $x_1, x_2\in \bT$ and $\alpha, \epsilon>0$ we set,
 $$
 h(x)= \left\{
 \begin{array}{ll}
  \alpha(x- x_1)& x\in (x_1-\epsilon, x_1+\epsilon),\\
  \alpha \epsilon & x\in (x_1+\epsilon, x_2-\epsilon),\\
  -\alpha(x- x_2)& x\in (x_2-\epsilon, x_2+\epsilon),\\
  -  \alpha \epsilon & x\in \bT \setminus( (x_1-\epsilon, x_2+\epsilon).
 \end{array}
\right.
$$
Of course, we assume that $2\epsilon< |x_1 -x_2|$. By definition, $h \in  W^{1,1}(\bT)$. In our notation we suppress the dependence of $h$ on $x_1$, $x_2$ $\alpha, \epsilon$.

We stick $w = u + h$ into formula (\ref{kpod}). The result is
\begin{eqnarray}\label{rpod1}
 &&\int_{x_1-\epsilon}^{x_1+\epsilon} W(u_x(s) + \alpha) -W(u_x(s))\,ds +
 \int_{x_2-\epsilon}^{x_2+\epsilon} W(u_x(s) - \alpha) -W(u_x(s)) \,ds\nonumber\\ &\ge&
 \alpha \int_{x_1-\epsilon}^{x_1+\epsilon} \xi(s)\,ds - \alpha 
 \int_{x_2-\epsilon}^{x_2+\epsilon} \xi(s)\,ds.
\end{eqnarray}
For each $\alpha>0$ there is a full measure set $A\subset \bT$ such that for all $y\in A$ we have
$$
\lim_{\epsilon\to 0} \int_{y-\epsilon}^{y+\epsilon} W(u_x(s) + \alpha) -W(u_x(s))\,ds
= W(u_x(y) + \alpha) -W(u_x(y)).
$$
We take any sequence $0<\alpha_k$ converging to zero and the corresponding set $A_{\alpha_k}$. Subsequently, we take any $x_1,$ $x_2\in A_0 = \bigcap_{k=1}^\infty  A_{\alpha_k}$. Then, we divide both sides of (\ref{rpod1}) by $2\epsilon$ and pass to the limit. In this way we obtain,
$$
W(u_x(x_1) + \alpha_k) -W(u_x(x_1)) +  W(u_x(x_2) - \alpha_k) -W(u_x(x_2)) \ge
\alpha_k(\xi(x_1) - \xi(x_2)),
$$
for $x_1, x_2\in A_0$. Now, we divide both sides of this inequality by $\alpha_k$ and pass to the limit. Since $W$ is a Lipschitz continuous function having one sided derivatives, then we obtain,
\begin{equation}\label{rpod2}
 W_p^+(u_x(x_1)) - W_p^-(u_x(x_2))\ge
\xi(x_1) - \xi(x_2).
\end{equation}
Here $W_p^+(y)$ (resp. $W_p^-(y)$) denotes the right (resp. left) derivative of $W$ at $y$.

Let us us suppose that there exists $x_1\in\bT$ such that 
\begin{equation}\label{rpod3}
\xi(x_1)> \max\{\omega: \omega\in \d W(u_x(x_1)\} \equiv W_p^+(u_x(x_1)).
\end{equation}
Since $\xi$ is continuous and set $A_0$ has the full measure so it is dense, we may assume that $x_1 \in A_0$.

We notice that (\ref{rpod2}) and (\ref{rpod3}) combined imply
$$
W_p^+(u_x(x_1)) - W_p^-(u_x(x_2))> W_p^+(u_x(x_1)) - \xi(x_2).
$$
Hence for all $x_2$ in $A_0$ we have
\begin{equation}\label{rpod4}
  \xi(x_2)> W_p^-(u_x(x_2)).
\end{equation}
A similar reasoning may be performed, when 
$$
\xi(x_2)< \min\{\omega: \omega\d W(u_x(x_2)\} \equiv W_p^-(u_x(x_2)).
$$
Let us notice that if $\xi$ satisfies (\ref{kpod}) and $b$ is a real constant, then $\xi -b$ satisfies (\ref{kpod}) too. Indeed, if $\psi$ is an element of $W^{1,1}(\bT)$, then
$$
\int_\bT (\xi -b) \psi_x \,dx = \int_\bT \xi \psi_x \,dx.
$$

Let us define 
$$
b_0 = \sup\{ \xi(x) - W^+_p(u_x(x)): \ x\in A_0\}.
$$ 
Due to continuity of $\xi$ and the linear growth of $W$ the number $b_0$ is finite. Since we assumed (\ref{rpod3}), then $b_0$ is positive. 

Let us consider shifts $\xi-b$, where $b\in(0,b_0)$. If for all such shifts we have that 
$$
\xi(x_1) -b > W^-_p(u_x(x_1)),\ \forall  x_1\in A_0,
$$
then due to continuity of $\xi$ we will have
$$
\xi(x_1) -b_0 \in d W(u_x(x_1)),\ \forall  x_1\in A_0
$$
hence our claim follows after redefining $\xi$.

If on the other hand there is $b\in(0,b_0)$ such that there is $x_2\in A_0$ such that 
$\xi(x_2) - b< W^-_p(u_x(x_1))$, then due to the definition of $b_0$ we have $\xi(x_1) -b > W^+_p(u_x(x_1))$. Thus, we reached a contradiction with (\ref{rpod4}). Our claim follows.
\end{proof}
 
\section{Integrability of the derivative of solutions to the evolution problem}\label{ser}
In  this section we study the integrability of the space derivative of solutions to the following evolution problem,
\begin{equation}\label{r1}
 \begin{array}{ll}
  u_t = (W_p( u_x) )_x, &(x,t)\in Q_T: =\bT\times (0,T), \\
  u(x, 0) = u_0(x), & x\in \bT.
 \end{array}
\end{equation}
We assume here the  periodic boundary conditions, but the same argument applies to the homogeneous Neumann  data. The initial value, $u_0$, is  in $W^{1,1}$.

The question we address here is as follows: let us suppose that $u_0\in W^{1,1}$, is it true that $u(t)\in W^{1,1}$ for a.e. $t>0$? We give an affirmative answer below. This means that in general, equation (\ref{r1}) does not create singularities like jumps. 


A relatively simple way to address the question of existence of solutions is by using the nonlinear semigroup theory by K\=omura. It is based on the observation that (\ref{r1}) is formally a gradient flow of $\cE$. For this purpose we have to consider $\bar\cE$, the lower semicontinuous envelope of $\cE$ defined by formula (\ref{defbe}), see also (\ref{zero}),  in place of $\cE$. 
Here it is.
\begin{proposition}\label{semig}
Let us suppose that $W$ is convex and even, with linear growth, i.e. (\ref{gr}) holds. 
If $u_0\in BV(\bT)$, then there is a unique function $u:[0,\infty)\to L^2(\bT)$, such that\\
(1) for all $t>0$ we have $u(t)\in D(\d\bar\cE (u(t)))$;\\
(2) $u\in L^\infty (0,\infty; BV(\bT))$;\\
(3) $-\frac{du}{dt}\in \d\bar\cE(u(t))$ a.e. on $(0,\infty)$;\\
(4) $u(0) = u_0$.

In addition, $u$ has a right derivative for all $t\in (0,\infty)$ and
$$
\frac{d^+ u}{dt} + (\d\bar\cE(u(t)))^o =0,\qquad \hbox{for }a.e. \ t\in(0,\infty),
$$
where $(\d\bar\cE(u(t)))^o$ is the minimal section of $\d\bar\cE(u(t)))$, i.e. the element of $\d\bar\cE(u(t)))$ with the smallest norm. 
\end{proposition}
{\it Proof.} Due to convexity and lower semicontinuity of $\bar\cE$ with respect to the $L^2$ convergence, this fact follows immediately from 
\cite[Theorem 3.2]{brezis}.

\bigskip
This Theorem has a drawback. Namely, in order to make this result meaningful, we have to identify the subdifferential of $\bar\cE$.
We would like to contrast it with our main result, stated below.

\begin{theorem}\label{main1}
 Let us suppose that $W:\bR\to\bR$ is convex with linear growth, (\ref{gr}) holds and $u_0\in W^{1,1}$. Then,  there is a unique weak solution to (\ref{r1}), i.e.
there are  $u\in L^\infty(0,\infty; W^{1,1}(\bT)),$  $u_t\in L^2(0,\infty;L^2(\bT))$ and 
$\xi\in L^\infty(0,\infty; L^\infty(\bT))$
 such that
\begin{equation}\label{rweak}
\int_\bT (u_t(x,t)\vfi(x)  + \xi(x,t)\vfi_x(x))\,dx =0 \qquad a.e.\ t>0\quad 
\forall \vfi\in C^\infty(\bT)
\end{equation}
and $\xi(x,t)\in \partial W(u_x(x,t))$ for a.e. $(x,t)\in Q_T$. 
In particular, 
 $\cE(u(t)) = \bar \cE(u(t))$. Moreover, $\cE(u(t)) \le \cE(u_0)$.
\end{theorem}
The proof of this result will be performed in several steps. Before we engage in it, we will make a few comments. When we constructed, by approximation, the solutions to the elliptic problem (\ref{iE}), we had to resolve the following issues:\\
1) Making sure that the limiting function $u$ has the desired integrability properties, see (\ref{kees}).\\ 
2) Making sure that the limiting function $u$ is indeed a weak solution, i.e. the limit $\xi$ of $\xi^\ep W^\ep_p( u^\ep_x)$ is indeed an element of $\d W( u_x)$. We used for this Minty's trick.

In order to resolve these issues for the parabolic problem (\ref{r1}), we will proceed in a similar way, i.e. we will consider an auxiliary problem, whose initial conditions are regular,
\begin{equation}\label{r1e}
 \begin{array}{ll}
  u^\ep_t = (W_p( u^\ep_x) )_x, 
  &(x,t)\in Q_T,\\
  u^\ep(x, 0) = (u_0*\phi_\ep)(x), & x\in \bT,
 \end{array}
\end{equation}
where $u_0*\rho_\ep$ 
is a convolution with the standard mollifying kernel $\rho_\ep$. 

We recall the basic existence result for (\ref{r1e}).
\begin{proposition}(\cite[Theorem 1]{mury})\label{pmury}\\
 Let us assume that $W$ satisfies hypotheses of Theorem \ref{main1}. If  $u_0\in BV(\bT)$ and $(u_0)_x\in BV(\bT)$, then there exists a unique weak solution $u$ to  (\ref{r1e}). More precisely,
$u_x\in L^\infty(0,T; BV(\bT))$, $u_t\in L^2(Q_T)$ and  there is $\xi \in L^2(0,T; W^{1,2}(\bT))$ satisfying the (\ref{rweak}).
Moreover, $\xi(x,t)\in \d W(u_x)$ for a.e. $(x,t)\in Q_T$. 
\end{proposition}
In order to underline the dependence of solutions, obtained in this way, on the mollifying parameter $\ep$, we will denote them by $u^\ep$ and $\xi^\ep$. 
However, the result above is not sufficient for establishing estimates on solutions, which require prior regularization of $W$. For this purpose, we have to recall the problem, which led to Proposition \ref{pmury}, see \cite{mury},
\begin{equation}\label{r1reg}
 \begin{array}{ll}
  u^{\ep,\gamma}_t = (W_p^\gamma( u^{\ep,\gamma}_x) )_x + \gamma u^{\ep,\gamma}_{xx}, 
  &(x,t)\in Q_T,\\
  u^{\ep,\gamma}(x, 0) = (u_0*\rho_\ep)(x), & x\in \bT,
 \end{array}
\end{equation}
where $W^\gamma = W * \rho_\gamma$ and $\rho_\gamma$ is the standard mollifier kernel. By the classical theory, see \cite{LSU}, solutions $u^{\ep,\gamma}$ to (\ref{r1reg}) are smooth. 

We wish to proceed as in the proof of Theorem \ref{main2}. For this purpose, we fix $\Phi$ corresponding to $u_{0,x}$, see Lemma \ref{fi}. With its help we will establish additional estimates of solutions to (\ref{r1e}).

\begin{lemma}\label{est4}
 Let us suppose that $u^\ep$ is a unique weak solution to (\ref{r1e}) and $\Phi$ corresponding to $u_{0,x}$ is given by  Lemma \ref{fi}. Then,
 $$
\cG(u_x^{\ep}(\cdot,t)) \le \cG(\frac d{dx}\left(u_{0}^{\ep}\right))\le \cG(\frac d{dx}u_0).
$$
\end{lemma}
{\it Proof.}
We 
multiply both sides of (\ref{r1reg}) by $\Phi''(u_x^{\ep,\gamma})u^{\ep,\gamma}_{xx}$ and integrate over $\bT$ to obtain,
$$
\int_\bT u^{\ep,\gamma}_t (\Phi'(u_x^{\ep,\gamma}))_{x}\,dx =
\int_\bT (W^\gamma_{pp}(u_x^{\ep,\gamma})+\gamma)\Phi''(u_x^{\ep,\gamma}) |u^{\ep,\gamma}_{xx}|^2\,dx \ge 0.
$$
Positivity of the right-hand-side (RHS) is guaranteed by convexity of $W^\gamma$ and $\Phi$. Integration by parts of the left-hand-side (LHS) above yields,
$$
\frac d{dt}\int_\bT \Phi(u_x^{\ep,\gamma})\,dx \le 0,
$$
where the boundary terms dropped out due to the periodic boundary conditions. 

After integrating in time over $(0,T)$ and recalling the definition of $\cG$ we obtain,
$$
\cG(u_x^{\ep,\gamma}(\cdot,t)) \le \cG(u_{0,x}^{\ep}).
$$
We know from \cite{mury} that 
\begin{equation}\label{zbieg}
 u_x^{\ep,\gamma}\hbox{ converges to }u_x^{\ep}\hbox{ strongly in }L^p(0,T; L^q(\bT)),\ p\ge 1
 \hbox{ and a.e. in }Q_T,
\end{equation}
thus
$$
\cG(u_x^{\ep}(\cdot,t)) \le \cG(u_{0,x}^{\ep}).
$$
Since $\Phi$ is convex, then Jensen inequality gives us 
$$
\cG(u_{0,x}^{\ep}) \le \cG(u_{0,x}). \eqno\Box
$$
Now,
we want to pass to the limit with $\ep$, we need further estimates for this purpose. 
\begin{lemma}\label{est5}
Let us suppose that $u^\ep$ is a unique weak solution to (\ref{r1e}), then
\begin{equation}\label{est3}
\int_{Q_T} (u^\ep_t(x,t))^2\,dxdt + \int_\bT W(u^\ep_x(x,t)\,dx \le \int_\bT W(u^\ep_{0,x}(x)\,dx.
\end{equation}
\end{lemma}
\begin{proof}
We multiply eq. (\ref{r1reg}) by $u^{\ep,\gamma}_t$ and integrate over $Q_T$. Integrating by parts the RHS yields,
$$
\int_{Q_T} |u^{\ep,\gamma}_t|^2 \,dxdt + \int_{Q_T} \frac \d{\d t}\left(\frac \gamma2 |u^{\ep,\gamma}_x|^2 + W^\gamma(u_x^{\ep,\gamma})\right)\,dxdt =0.
$$
Performing the integration over $(0,T)$ leads us to,
$$
\int_{Q_T} |u^{\ep,\gamma}_t|^2 \,dxdt +
\int_\bT \left(\frac\gamma2 |u^{\ep,\gamma}_x(x,t)|^2 + W^\gamma(u_x^{\ep,\gamma}(x,t))\right)\,dx
=\int_\bT \left(\frac\gamma2 |u^{\ep}_{0,x}(x)|^2 + W^\gamma(u_{0,x}^{\ep}(x))\right)\,dx
$$
The RHS goes to $\int_\bT W^\gamma(u_{0,x}^{\ep}(x))\,dx$ as $\gamma \to 0$. We may drop 
$\int_\bT \frac\gamma2 |u^{\ep,\gamma}_x(x,t)|^2\,dx$ on the LHS.

The lower semicontinuity of the $L^2$ norm yields
$$
\varliminf_{\gamma\to0^+}\int_{Q_T} |u^{\ep,\gamma}_t|^2 \,dxdt \ge 
\int_{Q_T} |u^{\ep}_t|^2 \,dxdt.
$$
Now, when we regularize $W$, then we notice that the averaging of a convex function, performed in the convolution gives us $W(p) \le W^\ep(p)$ for all $p\in \bR$. As a result we arrive at
$$
 \int_\bT W(u^{\ep,\gamma}_x(x,t)) \le \int_\bT W^\gamma(u^{\ep,\gamma}_x(x,t)) \le M.
$$
We again use (\ref{zbieg})
to conclude that
$$
\lim_{\gamma\to0^+}\int_\bT W(u_x^{\ep,\gamma})(x,t)\,dx = \int_\bT W(u_x^{\ep})(x,t)\,dx \qquad a.e.\ t>0.
$$
Combining these gives the desired result.
\end{proof}

We notice that Lemma \ref{est5} immediately implies that
$$
u^\ep_t \rightharpoonup u_t\qquad\hbox{in } L^2(Q_T)\qquad\hbox{as }\ep \to 0.
$$
We know that $\xi^\ep$ postulated by Proposition \ref{pmury} satisfies
$$
\xi^\ep(x,t) \in \d W(u_x^\ep(\cdot, t)) \subset [-W^\infty, W^\infty].
$$
Here, the last inclusion is obtained by the argument, which gave us (\ref{xid}).

Hence, we deduce that there is a subsequence (not relabeled) such that
\begin{equation}\label{zb-e}
 \xi^\ep\rightharpoonup \xi \hbox{ in } L^2(Q_T)
\quad\hbox{and} \quad \xi^\ep
\stackrel{*}{\rightharpoonup} \xi \hbox{ in } L^\infty(Q_T).
\end{equation}
Using the  argument  from \cite[Theorem 2.1, page 2292]{mury-non}  we can show that
$$
\xi^\ep( 
\cdot, t)
\stackrel{*}{\rightharpoonup} \xi (\cdot, t)\hbox{ in } L^\infty(\bT)
\qquad \hbox{for } a.e. \ t>0.
$$
We may repeat the argument of 
\cite{mury}, \cite{nr} to claim that 
\begin{equation}\label{aubin}
 u^\ep\hbox{ converges to }u \hbox{ in } L^p(0,T; L^q(\Omega)),\quad p,q \in (1,\infty),
\end{equation}
hence $\| u^\ep(\cdot, t) - u(\cdot, t)\|_{L^q} \to 0$ for a.e. $t>0$. However, the key issue is convergence of $u_x^\ep$.


We notice that due to 
Lemma \ref{est4} and Lemma \ref{pettis},
we can select a subsequence 
$\{u^{\ep_k}_x\}_{k=1}^\infty$ such that $u^{\ep_k}_x$ converges weakly in $L^1(Q_T)$ to $u_x$ and, if we fix $t>0$, there is a subsequence (not relabeled) such that 
$u^{\ep_k}_x(\cdot, t)$ converges weakly in $L^1(\Omega)$ to $u_x(\cdot, t)$. However, copying the argument from \cite[Theorem 2.1, page 2292]{mury-non} 
leads us to the following statement:
\begin{lemma}\label{lm1.3}
 There is a sequence $u^k$, $k\in \bN$ such that 
 $$
 u_x^k \rightharpoonup u_x \quad \hbox{in } L^1(Q_T)
 $$
 and
 for almost all $t>0$,
 $$
 u_x^k(\cdot, t) \rightharpoonup u_x (\cdot, t)\quad \hbox{in } L^1(\Omega).
 $$
\end{lemma}

Here is an immediate conclusion from this Lemma and Lemma \ref{1.4}:
\begin{corollary}
 If $u_x$ is a weak limit in $L^1$ of the sequence $u_x^n$, then
 $$
 \cG(u(\cdot, t)) \le M<\infty\quad\hbox{and}\quad\cE(u(\cdot, t)) \le \cE(u_0)
 \qquad \hbox{for a.e. }t>0. \eqno\Box
 $$
\end{corollary}


Now, we claim that $u$ with $\xi$ is a weak solution to (\ref{r1}). If we inspect (\ref{rweak}), the weak form of (\ref{r1}), and integrate it over $(0,T)$, assuming that $\phi\in C^\infty_0(Q_T)$, then we will see
\begin{equation}\label{wQt}
 \int_{Q_T} u^\ep_t(x,t) \phi (x,t) \,dxdt  + \int_{Q_T} \xi^\ep(x,t) \phi_x(x,t)\,dxdt =0.
\end{equation}
The stated above weak convergence of $u^\ep_t$ and $\xi^\ep$ gives us,
$$
\int_{Q_T} u_t(x,t) \phi (x,t) \,dxdt  + \int_{Q_T} \xi(x,t) \phi_x(x,t)\,dxdt =0.
$$
We can localize it by arguing like in \cite[Theorem 2.1, page 2292]{mury-non},
$$
\int_{\bT} u_t(x,t) \psi (x) \,dx  + \int_{\bT} \xi(x,t) \psi_x(x)\,dx =0 \qquad\hbox{for } a.e. t>0\quad \hbox{ and all }\psi\in C^\infty(\bT).
$$
We notice that since $C^\infty(\bT)$ is dense in $W^{1,1}(\bT)$ we can take $u$ as (\ref{wQt}).

Now, it remains to show that 
$\xi(x,t) \in \d W(u_x(x,t))$ for almost every $(x,t)\in Q_T$. Indeed, from the construction of $u^\epsilon$ we know that for any $w\in W^{1,1}$ and for a.e. $t>0$ we have
\begin{equation}\label{rdop-e}
\int_\bT W(w_x(x))\,dx  \ge 
\int_\bT \xi^\epsilon(x,t)(w_x(x) - u^\epsilon_x(x,t))\,dx +\int_\bT W(u^\epsilon_x(x,t))\,dx.
\end{equation}

In order to use (\ref{zb-e}) and Lemma \ref{lm1.3}  we multiply (\ref{rdop-e}) by $\psi\ge0$ and $\psi\in C^\infty_0(0,T)$ and integrate over $(0,T)$. We get,
$$
\int_{Q_T}\psi W(w_x)\,dxdt \ge  \int_{Q_T} \psi \xi^\epsilon (w_x- u^\epsilon_x)\,dxdt +\int_{Q_T} \psi W(u^\epsilon_x)\,dxdt.
$$
Due to Lemma \ref{1.4} $\varliminf_{n\to\infty}\int_{Q_T} \psi W(u^\epsilon_x)\,dxdt\ge \int_{Q_T} \psi W(u_x)\,dxdt$. 

If we use $u^\ep$ as a test function in (\ref{wQt}), then we reach,
$$
\int_{Q_T} \xi^\epsilon u_x^\epsilon \,dxdt= \int_{Q_T}  u^\epsilon_t u^\epsilon \,dxdt.
$$
Since $u^\ep \in W^{1,1}(Q_T)$, when $u^\ep $ converges strongly to $u$, (possibly after extracting a subsequence).
Obviously, 
$$
\lim_{n\to\infty}\int_{Q_T}  u^\epsilon_t u^\epsilon \,dxdt = 
\int_{Q_T}  u_t u \,dxdt .
$$
Thus, we have reached
$$
\int_{Q_T}\psi W(w_x)\,dxdt - \int_{Q_T} \psi W(u_x)\,dxdt \ge
\int_{Q_T} \psi(\xi w_x + u u_t)\,dxdt = \int_{Q_T} \psi\xi( w_x - u_x)\,dxdt,
$$
where we use (\ref{wQt}) again in the last equality.
Since $\psi\ge 0$ was arbitrary, then we deduce that
\begin{equation}\label{podpo}
 \int_{\bT}W(w_x)\,dx - \int_{\bT} W(u_x)\,dx \ge
\int_{\bT} \xi(w_x- u_x )\,dx . 
\end{equation}
Now, we apply Lemma \ref{lpod} to deduce 
that $\xi(x,t) \in\d W(u_x(x,t))$ a.e. in $Q_T$. 

Thus, we finished a construction of a weak solution to (\ref{r1}) satisfying the desired bound.
Now, we notice that the solution we constructed satisfies the properties stipulated by Proposition \ref{semig}, hence we deduce
uniqueness of solutions. 
This finishes the proof of Theorem \ref{main1}. \qed

We also notice that in fact we constructed in  Theorem \ref{main1} solutions in the sense of Proposition \ref{semig}. 

\subsection{Common properties of solutions}
Since we made rather weak assumptions on the nonlinearity $W$, we should not expect too many common features of solutions. The property, which draws attention, when we deal with the total variation flow is the finite stopping time of solutions, i.e. at some time instance the solution stops moving having reached a terminal state. In this section we will relate the finite stopping time to the lack of differentiability of $W$ at $p=0$. The behavior of $W$ for  large arguments does not seem to matter.

\begin{theorem}
 Let us suppose that $u_0\in W^{1,1}(\Omega)$ and $W$ is such that at all points $p$, the one-sided derivatives of $W$, at $p$ are greater or equal to $\alpha>0$. Then, for all $t\ge T_{ext}$, we have $u(t) \equiv \bar u_0$, where $\bar u_0 = \frac{1}{|\Omega|} \int_\Omega u_0\, dx$ and 
 $$
 T_{ext} \le C_p\|u_0\|_{L^2},
 $$
 and $C_p$ is the constant in the Poincar\'e inequality.
\end{theorem}
\begin{proof}
We notice that the average of solutions is preserved due to the boundary conditions. We denote this average by $\bar u$. We compute $\frac{d}{dt}\|u- \bar u\|^2_{L^2}$, while integrating by parts
\begin{eqnarray*}
 \frac12\frac{d}{dt} \int_\Omega|u(x,t) - \bar u|^2\,dx &=&
 \int_\Omega (u-\bar u) u_t\,dx = \int_\Omega (u-\bar u) \left(W_p(u_x)\right)_x\\
 &=& - \int_\Omega W_p(u_x) u_x  = -  \int_\Omega | W_p(u_x) | \sgn u_x \cdot u_x \,dx.
\end{eqnarray*}
We used here monotonicity of $W_p$, which implies that $ W_p(u_x) u_x  = | W_p(u_x) | | u_x|$. Hence,
\begin{eqnarray*}
 \frac12\frac{d}{dt} \int_\Omega|u(x,t) - \bar u|^2\,dx &\le&
 -  \int_\Omega \alpha |u_x|\,dx  \le - C_p^{-1} \| u- \bar u\|_{L^2}.
\end{eqnarray*}
Here, we used the Poincar\'e's inequality, $\| u- \bar u\|_{L^2}\le C_p \| u_x\|_{L^1}$.
We conclude that 
$$
\frac{d}{dt}\| u- \bar u\|_{L^2} \le -  C_p,
$$
what implies that $T_{ext} \le C_p\| u_0\|_{L^2}$.
\end{proof}

\section*{Acknowledgement}
A part of the research was conducted during the visits of the second author to the University of Tokyo whose hospitality and  support is thankfully acknowledged. PR enjoyed also a partial support of the EU IRSES program ``FLUX'' and the Polish Ministry of Science and Higher Education  grant number 2853/7.PR/2013/2.


\begin{thebibliography}{88}
%
\bibitem{AFP}
L.Ambrosio, N.Fusco, D.Pallara,
{\it Functions of bounded variation and free discontinuity problems}.
Oxford Mathematical Monographs. The Clarendon Press, Oxford
University Press, New York, 2000
%
%
\bibitem{bellettini}
G.Bellettini, M.Novaga, G.Orlandi, 
Eventual regularity for the parabolic minimal
surface equation,
{\it Discrete Contin. Dyn. Syst.}, {\bf 35} (2015), no. 12, 5711--5723.
%
\bibitem{brezis}
H. Brezis,
{\it Operateurs maximaux monotones et semi-groupes de contractions dans les espaces
de Hilbert},
North-Holland Mathematics Studies, No. 5. Notas de Matematica (50). 
North-Holland Publishing Co., Amsterdam-London, 1973
%
\bibitem{bulicek}
L. Beck, M. Bul\'i\v cek, J. M\'alek and E. S\"uli, On the existence of integrable solutions to nonlinear elliptic systems and variational problems with linear growth, {\it preprint.}
%
\bibitem{maringova}
L. Beck, M. Bul\'i\v cek, E.Maringov\'a,
Globally lipschitz minimizers for variational problems
with linear growth,
https://arxiv.org/abs/1609.07601
%
\bibitem{caselles2007}
V.Caselles, A.Chambolle, M.Novaga, The discontinuity set of solutions of the TV denoising problem and some extensions, {\it  Multiscale Model. Simul.}, {\bf 6} (2007), no. 3, 879-894.
%
\bibitem{caselles2011}
V.Caselles, A.Chambolle, M.Novaga, 
Regularity for solutions of the total variation denoising problem, {\it  Rev. Mat. Iberoam.}, {\bf  27} (2011), no. 1, 233-252.
%
\bibitem{LSU} O.A.Lady\v zenskaja,  V.A.Solonnikov,  N.N.Ural'ceva,
{\it Linear and Quasilinear Equations of Parabolic Type,}
Am. Math. Soc., 
  Providence, R. I., 1968 
  %
 \bibitem{mm} M.Matusik, P.Rybka,
Oscillating facets,  {\it Port. Math.}, {\bf  73}, (2016), 1-40.
%
\bibitem{mury} P.B.Mucha, P.Rybka,
Well-posedness of sudden directional diffusion equations, {\it Math. Meth. Appl. Sci.}, {\bf 36}, (2013), 2359-2370. 
 
 %
 \bibitem{mury-non}
 P.B. Mucha, P.Rybka, A caricature of a singular curvature flow in the plane, 
 {\it Nonlinearity,} {\bf  21}, (2008), 2281-2316. 
 %
 \bibitem{nr} A.Nakayasu, P.Rybka, Energy solutions to one-dimensional singular parabolic problems with BV data are viscosity solutions, http://arxiv.org/abs/1603.07502, to appear in "Mathematics for Nonlinear Phenomena: Analysis and Computation - Proceedings in Honor of Professor Yoshikazu Giga's 60th birthday", Springer Proceedings in Mathematics and Statistics.
 %
 \bibitem{RaoRen}
 M.M.Rao, Z.D.Ren,  {\it Theory of Orlicz spaces}. Monographs and Textbooks in Pure and Applied Mathematics, 146. Marcel Dekker, Inc., New York, 1991.
 %
\bibitem{rudin-etco} L.I.Rudin, S.Osher, E.Fatemi, Nonlinear
  total variation based noise removal argorithms, {\it Physica D},
  {\bf 60}, (1992), 259-268.
\end{thebibliography}
\end{document}